\documentclass[12pt,reqno,a4wide]{amsart}
\usepackage{amsfonts}
\usepackage{url}
\usepackage{listings}
\theoremstyle{plain}
\newtheorem{theorem}{Theorem}
\usepackage{tcolorbox}

\usepackage[osf,sc]{mathpazo}

\theoremstyle{definition}
\newtheorem{definition}{Definition}

\allowdisplaybreaks
\newcommand{\pp}{\textsc{PL}}
\newcommand{\bell}{\textup{B}}

\begin{document}
	
\title[Colored and plane partitions]{Formulas for the number of $k$-colored partitions and the number of plane partitions of $n$ in terms of the Bell Polynomials}

\author{Sumit Kumar Jha}
\address{International Institute of Information Technology\\
Hyderabad-500032, India}
\email{kumarjha.sumit@research.iiit.ac.in}


\begin{abstract}
We derive closed formulas for the number of $k$-coloured partitions and the number of plane partitions of $n$ in terms of the Bell polynomials. 
\end{abstract}

\maketitle

\section{Introduction}
\begin{definition}[$k$-colored partitions of $n$ \cite{Chern}]
A partition of a positive integer $n$ is a finite weakly decreasing sequence of positive integers $\lambda_{1}\geq \lambda_{2}\geq \cdots \geq \lambda_{r}>0$ such that $\sum_{i=1}^{r}\lambda_{i}=n$. The $\lambda_{i}$'s are called the \emph{parts} of the partition. Let $p(n)$ denote the number of partitions of $n$.\par 
A partition is said to be \emph{$k$-coloured} if each part can occur as $k$ colours. Let $p_{k}(n)$ denote the number of $k$ coloured partitions of $n$. For example, there are five $2$-coloured partitions of $2$:
$$
1+1, 1'+1, 1'+1', 2, 2'
$$
so $p_{2}(2)=5$.
\end{definition}
The generating function for $p_{k}(n)$ is given by
$$
\sum_{n=0}^{\infty}p_{k}(n)q^{n}=\prod_{j=1}^{\infty}\frac{1}{(1-q^{j})^{k}} \qquad (|q|< 1).
$$
The right hand side of the above equation is also the generating function for the \emph{multipartitions} which appear in many papers on the representations theory of Lie algebras, and some with applications in physics \cite{Andrews}.\par 
The unrestricted partition function $p(n)=p_{1}(n)$ can be computed using Hardy-Ramanujan-Rademacher formula with soft complexity $O(n^{1/2+o(1)})$  and very little overhead \cite{Johansson}.\par  In section 2, we obtain a formula for the number of $k$-colored partitions of $n$ in terms of the partial Bell polynomials which allows computing them efficiently. For example, using our formula we can calculate  
$$
p_{30}(200)=23945275792616100703623332622769220026826156718318470749445535353589
$$
which are the number of thirty colored partitions of $200$.
\begin{definition}[Plane partitions of $n$]
A \emph{plane partition} of $n$ is a two-dimensional array of integers $n_{i,j}$ that are non-increasing in both indices, that is, 
$$
{\displaystyle n _{i,j}\geq n _{i,j+1}\quad }  {\displaystyle \quad n _{i,j}\geq n _{i+1,j}\,} 
$$
and
$$
n=\sum_{i,j}n_{i,j}.
$$
Let $\pp(n)$ denote the number of plane partitions of $n$. For example, there are six plane partitions of $3$:
$$
{\begin{matrix}1&1&1\end{matrix}}\qquad {\begin{matrix}1&1\\1&\end{matrix}}\qquad {\begin{matrix}1\\1\\1&\end{matrix}}\qquad {\begin{matrix}2&1&\end{matrix}}\qquad {\begin{matrix}2\\1&\end{matrix}}\qquad {\begin{matrix}3\end{matrix}}$$
so $\pp(3)=6$.
\end{definition}
Macmohan \cite{Knuth} obtained the generating function of $\pp(n)$ as
$$
\sum_{n=0}^{\infty}\pp(n)\, q^{n}=\sum_{j=1}^{\infty}\frac{1}{(1-q^{j})^{j}}\qquad (|q|<1).
$$
In section $4$, we obtain an expression for $\pp(n)$ in terms of the complete Bell polynomials and sum of squares of divisors of $n$. This also gives a determinant expression for $\pp(n)$. For example, we can calculate
$$
\pp(700)=1542248695905922088013690041381656661664744761954709483748320717869.
$$
\begin{definition}[Partial Bell polynomials]
For $n$ and $k$ non-negative integers, the partial $(n,k)$th partial Bell polynomials in the variables $x_{1},x_{2},\dotsc,x_{n-k+1}$ denoted by $B_{n,k}\equiv\bell_{n,k}(x_1,x_2,\dotsc,x_{n-k+1})$ \cite[p. 206]{Comtet} can be defined by
\begin{equation*}
B_{n,k}(x_1,x_2,\dotsc,x_{n-k+1})=\sum_{\substack{1\le i\le n,\ell_i\in\mathbb{N}\\ \sum_{i=1}^ni\ell_i=n\\ \sum_{i=1}^n\ell_i=k}}\frac{n!}{\prod_{i=1}^{n-k+1}\ell_i!} \prod_{i=1}^{n-k+1}\Bigl(\frac{x_i}{i!}\Bigr)^{\ell_i}.
\end{equation*}
\end{definition}
The partial Bell polynomials have the following generating function:
$$
{\displaystyle {\frac {1}{k!}}\left(\sum _{j=1}^{\infty }x_{j}{\frac {t^{j}}{j!}}\right)^{k}=\sum _{n=k}^{\infty }B_{n,k}(x_{1},\ldots ,x_{n-k+1}){\frac {t^{n}}{n!}},\qquad k=0,1,2,\ldots }.
$$
The partial Bell polynomials can also be computed efficiently by a recurrence relation \cite{Hamed}:
$$
{\displaystyle B_{n,k}=\sum _{i=1}^{n-k+1}{\binom {n-1}{i-1}}x_{i}B_{n-i,k-1},}
$$
where
\begin{align*}
    {\displaystyle B_{0,0}=1}\\
 {\displaystyle B_{n,0}=0{\text{ for }}n\geq 1;}\\
 {\displaystyle B_{0,k}=0{\text{ for }}k\geq 1.} 
\end{align*}
Cvijovi\'{c} \cite{Jkovic} gives the following formula for calculating these polynomials
\begin{align}
\label{explicit}
 B_{n, k + 1}  =  & \frac{1}{(k+1)!} \underbrace{\sum_{\alpha_1\,=  k}^{n-1} \, \sum_{\alpha_2\,=  k-1}^{\alpha_1-1} \cdots  \sum_{\alpha_k\, = 1}^{\alpha_{k-1}-1} }_{k }
\overbrace{\binom{n}{\alpha_1}  \binom{\alpha_1}{\alpha_2} \cdots  \binom{\alpha_{k-1}}{\alpha_k}}^{k}\nonumber
\\
&\cdot x_{n-\alpha_1} x_{\alpha_1 -\alpha_2} \cdots x_{\alpha_{k-1}-\alpha_k} x_{\alpha_k} \qquad(n\geq k+1, k\,=1, 2, \ldots)
\end{align}
The $n$th \emph{complete Bell polynomials} are defined as 
$$
B_{n}(x_{1},x_{2},\cdots,x_{n})=\sum_{k=0}^{n} B_{n,k}(x_1,x_2,\dotsc,x_{n-k+1}).
$$
The complete Bell polynomials can be recurrently defined as
$$
 {\displaystyle B_{n+1}(x_{1},\ldots ,x_{n+1})=\sum _{i=0}^{n}{n \choose i}B_{n-i}(x_{1},\ldots ,x_{n-i})x_{i+1}}
$$
with the initial value   ${\displaystyle B_{0}=1}$.
We would frequently use \emph{Fa\`{a} di Bruno's formula} \cite[p. 134]{Comtet} which is
\begin{equation}
\label{faa}
{d^n \over dq^n} f(g(q)) = \sum_{l=1}^n f^{(l)}(g(q))\cdot B_{n,l}\left(g'(q),g''(q),\dots,g^{(n-l+1)}(q)\right).
\end{equation}
\section{Formula for $k$-coloured partitions of $n$}
\begin{theorem}
For all positive integers $n,k$
$$
p_{k}(n)=\frac{1}{n!}\sum_{l=0}^{n}(-1)^{l}\, (k)^{(l)}\, B_{n,l}(\lambda_{1},\lambda_{2},\cdots,\lambda_{n-l+1})
$$
where 
$$
\normalfont
\lambda_{i}=
\begin{cases}
(-1)^{m}\, i! & \text{if $i=\frac{m(3m+1)}{2}$ for some positive integer $m$}\\
(-1)^{m}\, i!& \text{if $i=\frac{m(3m-1)}{2}$ for some positive integer $m$}\\
0&\text{otherwise}
\end{cases}
,
$$
where $(k)^{(l)}=k\, (k+1)\cdots (k+l-1)$ represents the rising factorial.
\end{theorem}
\begin{proof}
We recall the Euler's pentagonal number theorem \cite[Equation 7.8]{Fine}
$$
g(q):=\prod_{j=1}^{\infty}(1-q^{j})=\sum_{n=\infty}^{\infty}(-1)^{n}\, q^{\frac{3n^2+n}{2}}.
$$
Let $f(q)=q^{-k}$. Then the Fa\`{a} di Bruno's formula \eqref{faa} gives
$$
{d^n \over dq^n} g(q)^{-k} = \sum_{l=1}^n (-1)^{l}\, \frac{(k)^{(l)}}{(g(q))^{l+k}}\, B_{n,l}\left(g'(q),g''(q),\dots,g^{(n-l+1)}(q)\right).
$$
Letting $q\rightarrow 0$ in the above we obtain our result readily.
\end{proof}
\section{Another expression for $p_{k}(n)$}
\begin{theorem}
For all positive integers $n,k$
$$
p_{k}(n)=\frac{B_{n}( k\, \sigma(1),1!\, k\, \sigma(2),\cdots,k\,(n-1)!\, \sigma(n))}{n!},
$$
where $\sigma(n)=\sum_{d|n}d$.
\end{theorem}
\begin{proof}
Let 
$$
h(q):=\prod_{j=1}^{\infty}\frac{1}{(1-q^{j})^{k}}.
$$
Then we have
\begin{align*}
    \log{h(q)}&=-\sum_{j=1}^{\infty}k\, \log(1-q^{j})\\
    &=\sum_{j=1}^{\infty}\sum_{l=1}^{\infty}k\,\frac{q^{lj}}{l}\\
    &=\sum_{n=1}^{\infty}k\, q^{n}\sum_{d|n}d^{-1}.\\
    &=\sum_{n=1}^{\infty}\frac{q^{n}\,k}{n}\sum_{d|n}d.
\end{align*}
Let $f(q)=e^{q}$ and $g(q)=\log{h(q)}$ in Fa\`{a} di Bruno's formula \eqref{faa} gives
$$
{d^n \over dq^n} h(q) = \sum_{l=1}^n h(q)\, B_{n,l}\left(g'(q),g''(q),\dots,g^{(n-l+1)}(q)\right).
$$
Letting $q\rightarrow 0$ in the above gives us our result readily.
\end{proof}
The determinant expression for the complete Bell polynomials \cite[Theorem 2.1]{Xu} gives us
$$
p_{k}(n)=\frac{1}{n!}\det\begin{bmatrix}
k\,\sigma(1)  & k\, \sigma(2)  & k\,\sigma(3)  & k\,\sigma(4) & \cdots & \cdots & k\, \sigma(n) \\  \\
-1 & k\,\sigma(1) & k\,\sigma(2)  & k\,\sigma(3) &  \cdots & \cdots &
k\,\sigma(n-1) \\  \\
0 & -2 & k\,\sigma(1) &  k\,\sigma(2) & \cdots & \cdots &
k\,\sigma(n-2) \\  \\
0 & 0 & -3 & k\,\sigma(1) & \cdots  & \cdots & k\,\sigma(n-3) \\  \\
0 & 0 & 0 & -4 & \cdots & \cdots & k\,\sigma(n-4) \\  \\
\vdots & \vdots & \vdots &  \vdots & \ddots & \ddots & \vdots  \\  \\
0 & 0 & 0 & 0 & \cdots & -(n-1) & k\,\sigma(1) \end{bmatrix}.
$$
Thus
$$
p_{2}(2)=\frac{1}{2}\det
\begin{bmatrix}
2 & 6\\
-1 & 2
\end{bmatrix}
=5.
$$
\section{Formula for Plane partitions of $n$}
\begin{theorem}
For all positive integers $n$
$$
\pp(n)=\frac{B_{n}( \sigma_{2}(1),1!\, \sigma_{2}(2),\cdots,(n-1)!\, \sigma_{2}(n))}{n!},
$$
where $\sigma_{2}(d)=\sum_{d|n}d^{2}$.
\end{theorem}
\begin{proof}
Let 
$$
h(q):=\prod_{j=1}^{\infty}\frac{1}{(1-q^{j})^{j}}.
$$
Then we have
\begin{align*}
    \log{h(q)}&=-\sum_{j=1}^{\infty}j\, \log(1-q^{j})\\
    &=\sum_{j=1}^{\infty}\sum_{l=1}^{\infty}j\,\frac{q^{lj}}{l}\\
    &=\sum_{n=1}^{\infty}n\, q^{n}\sum_{d|n}d^{-2}.\\
    &=\sum_{n=1}^{\infty}\frac{q^{n}}{n}\sum_{d|n}d^{2}.
\end{align*}
Let $f(q)=e^{q}$ and $g(q)=\log{h(q)}$ in Fa\`{a} di Bruno's formula \eqref{faa} gives
$$
{d^n \over dq^n} h(q) = \sum_{l=1}^n h(q)\, B_{n,l}\left(g'(q),g''(q),\dots,g^{(n-l+1)}(q)\right).
$$
Letting $q\rightarrow 0$ in the above gives us our result readily.
\end{proof}
We immediately have the following determinant formula for $\pp(n)$ \cite[Theorem 2.1]{Xu}
$$
\pp(n)=\frac{1}{n!}\det\begin{bmatrix}
\sigma_{2}(1)  & \sigma_{2}(2)  & \sigma_{2}(3)  & \sigma_{2}(4) & \cdots & \cdots & \sigma_{2}(n) \\  \\
-1 & \sigma_{2}(1) & \sigma_{2}(2)  & \sigma_{2}(3) &  \cdots & \cdots &
\sigma_{2}(n-1) \\  \\
0 & -2 & \sigma_{2}(1) &  \sigma_{2}(2) & \cdots & \cdots &
\sigma_{2}(n-2) \\  \\
0 & 0 & -3 & \sigma_{2}(1) & \cdots  & \cdots & \sigma_{2}(n-3) \\  \\
0 & 0 & 0 & -4 & \cdots & \cdots & \sigma_{2}(n-4) \\  \\
\vdots & \vdots & \vdots &  \vdots & \ddots & \ddots & \vdots  \\  \\
0 & 0 & 0 & 0 & \cdots & -(n-1) & \sigma_{2}(1) \end{bmatrix}.
$$
For example, using the above we have
$$
\pp(3)=\frac{1}{6}\det
\begin{bmatrix}
1 & 5 & 10 \\
-1 & 1 & 5 \\
0 & -2 & 1
\end{bmatrix}
=6.
$$


\begin{thebibliography}{99}
\bibitem{Chern}
S. Chern, S. Fu \& D. Tang, Some inequalities for $k$-colored partition functions, Ramanujan J 46, 713--725 (2018).
\bibitem{Andrews}
Andrews G.E. (2008) A Survey of Multipartitions Congruences and Identities. In: Surveys in Number Theory. Developments in Mathematics (Diophantine Approximation: Festschrift for Wolfgang Schmidt), vol 17. Springer, New York, NY.
\bibitem{Johansson}
Johansson, F. (2012). Efficient implementation of the Hardy–Ramanujan–Rademacher formula. LMS Journal of Computation and Mathematics, 15, 341-359.
\bibitem{Knuth}
Knuth, D. (1970). A Note on Solid Partitions. Mathematics of Computation, 24(112), 955-961
\bibitem{Comtet} L. Comtet, {\em Advanced Combinatorics: The Art of Finite and Infinite Expansions}, D. Reidel Publishing Co.,  Dordrecht, 1974.
\bibitem{Hamed}
H. Taghavian, A fast algorithm for computing Bell polynomials based on index break-downs using prime factorization, preprint, 2020. Available at \url{https://arxiv.org/abs/2004.09283}.
\bibitem{Jkovic}
D. Cvijovi\'{c}, New identities for the partial Bell polynomials, Applied mathematics letters, \textbf{24} (2011), 1544--1547.
\bibitem{Xu}
A. Xu \& Z. Cen, On a
$q$-analogue of Fa\`{a} di Bruno’s determinant formula, Discrete Mathematics, \textbf{6} (2011), 387--392.
\bibitem{Fine}
N. J. Fine, \emph{Basic Hypergeometric Series and Applications}, American Mathematical Soc., 1988.
\end{thebibliography}
\end{document}